\documentclass[12pt]{amsart}
\usepackage{amssymb}
\usepackage[mathscr]{eucal}
\usepackage{parskip}
\usepackage{xcolor}
\usepackage{tikz}
\usetikzlibrary{calc}
\usepackage{tikz-cd}
\usepackage{hyperref}
\numberwithin{equation}{section}
\usepackage[margin=2.9cm]{geometry}
\usepackage{booktabs}
\usepackage{enumerate}
\usepackage{bm}

\usepackage{listings}
\theoremstyle{plain}
\newtheorem{thm}{Theorem}[section]
\newtheorem{lem}[thm]{Lemma}
\newtheorem{cor}[thm]{Corollary}
\newtheorem{prop}[thm]{Proposition}
\newtheorem{conj}[thm]{Conjecture}

 \theoremstyle{definition}
\newtheorem{defn}[thm]{Definition}
\newtheorem{rem}[thm]{Remark}
\newtheorem{ex}[thm]{Example}
\newtheorem{notn}[thm]{Notation}

\newcommand{\disc}{\operatorname{disc}}
\newcommand{\gdisc}{\operatorname{gdisc}}
\newcommand{\dupl}{\operatorname{dupl}}
\newcommand{\res}{\operatorname{res}}
\newcommand{\tol}{\operatorname{tol}}
\newcommand{\Char}{\operatorname{char}}
\newcommand{\sep}{\mathrm{sep}}
\newcommand{\GW}{\operatorname{GW}}
\renewcommand{\d}{\operatorname{d}}
\newcommand{\lc}{\operatorname{lc}}
\newcommand{\tc}{\operatorname{tc}}

\newcommand{\mb}[1]{\mathbb{#1}}

\makeatletter
\newcommand*\rel@kern[1]{\kern#1\dimexpr\macc@kerna}
\newcommand*\widebar[1]{%
  \begingroup
  \def\mathaccent##1##2{%
    \rel@kern{0.8}%
    \overline{\rel@kern{-0.8}\macc@nucleus\rel@kern{0.2}}%
    \rel@kern{-0.2}%
  }%
  \macc@depth\@ne
  \let\math@bgroup\@empty \let\math@egroup\macc@set@skewchar
  \mathsurround\z@ \frozen@everymath{\mathgroup\macc@group\relax}%
  \macc@set@skewchar\relax
  \let\mathaccentV\macc@nested@a
  \macc@nested@a\relax111{#1}%
  \endgroup
}
\makeatother
\makeatletter
\@namedef{subjclassname@2020}{\textup{2020} Mathematics Subject Classification}
\makeatother

\begin{document}
%%%%%%%%%%%%%%%%%%%%%%%%%%%%%%%%%%%%%%%%
\title{Tolerants}

\author[Adhikari]{Swechchha Adhikari}
\address{Department of Mathematics \\ Brigham Young University} 
\email{adhikar6@student.byu.edu}

\author[Hall]{Brent Hall}
\address{Department of Mathematics \\ Brigham Young University} 
\email{brentah@student.byu.edu}

\author[McKean]{Stephen McKean}
\address{Department of Mathematics \\ Brigham Young University} 
\email{mckean@math.byu.edu}
\urladdr{shmckean.github.io}

%%%%%%%%%%%%%%%%%%%%%%%%%%%%%%%%%%%%%%%%
\begin{abstract}
    We study a generalization of the discriminant of a polynomial, which we call the tolerant. The tolerant differs by multiplication by a square from the duplicant, which was discovered in recent work on $\mathbb{P}^1$-loop spaces in motivic homotopy theory. We show that the tolerant is rational by deriving a formula in terms of discriminants. This allows us to formulate a conjectural unstable Poincaré--Hopf formula over an arbitrary locus of points. We also show that the tolerant satisfies many of the same properties as the discriminant. A notable difference between the two is that the discriminant is inversion invariant for all polynomials, whereas the tolerant is only inversion invariant on a proper multiplicative subset of polynomials.
\end{abstract}

\maketitle

\section{Introduction}
In topology, May's recognition principle characterizes finite loop spaces as algebras over the little cubes operad \cite{May72}. An analogous recognition principle for finite $\mb{P}^1$-loop spaces in the motivic homotopy category has remained an open question for a quarter century \cite{Voe00}. A motivic finite loop space recognition principle would have spectacular applications, such as potentially simplifying Asok--Bachmann--Hopkins's seminal work on the motivic Freudenthal suspension theorem \cite{ABH24}, which in turn resolved an important conjecture of Murthy \cite[p.~173]{Mur99}.

Recently, Igieobo--McKean--Sanchez--Taylor--Wickelgren discovered a family of configuration space operations on the $\mb{P}^1$-loop space $\Omega_{\mb{P}^1}\mb{P}^1$ \cite{IMSTW24}. These operations present a tantalizing shadow of a potential operadic structure on $\mb{P}^1$-loop spaces. In this article, we will prove some algebraic properties of these configuration space operations. To begin, we need a definition.

\begin{defn}\label{def:duplicant}
    Let $k$ be a field with algebraic closure $\overline{k}$. Let $f=a_nx^n+\ldots+a_0$ be a polynomial over $k$. Write 
    \[f=a_n\prod_{i=1}^s(x-r_i)^{m_i}\]
    for the factorization of $f$ over $\overline{k}$, where $r_1,\ldots,r_s$ are distinct. The \emph{duplicant} of $f$ is defined as
    \[\dupl(f):=a_n^{2n}\prod_{1\leq i<j\leq s}(r_i-r_j)^{2m_im_j}.\]
\end{defn}

Note that when $m_i=1$ for all $i$ (i.e.~when $f$ is separable), then the duplicant is equal to $a_n^2$ times the discriminant of $f$. However, in contrast to the discriminant, the duplicant does not vanish when $f$ has a repeated root.

The duplicant lies at the heart of the operations discussed in \cite[Theorem 1.1]{IMSTW24}. Given a rational configuration $D:=\{r_1,\ldots,r_n\}\subseteq\mb{A}^1(k)$ and a family of pointed maps $f_1,\ldots,f_n:\mb{P}^1\to\mb{P}^1$ (that is, $n$ points in the loop space $\Omega_{\mb{P}^1}\mb{P}^1$), we obtain a natural sum
\[\sum_D(f_1,\ldots,f_n):\mb{P}^1\to\mb{P}^1\]
that is determined by the duplicant of $\prod_{i=1}^n(x-r_i)^{\deg(f_i)}$.

It is desireable, both to better understand a possible operadic structure on $\Omega_{\mb{P}^1}\mb{P}^1$ and for applications to enumerative geometry, to remove the assumption that the points in the configuration $D$ be rational. This article contributes to that goal by proving, among other things, that the duplicant of any polynomial over $k$ is again an element of $k$. This proof will pass through the following definition.

\begin{defn}
    Assume the notation of Definition~\ref{def:duplicant}. The \emph{tolerant}\footnote{The term tolerant was suggested by the first two authors because the tolerant \emph{tolerates} repeated roots, whereas the discriminant \emph{discriminates} against repeated roots.} of $f$ is defined as
    \[\tol(f):=\frac{\dupl(f)}{a_n^2}=a^{2n-2}_n\prod_{1\leq i<j\leq s}(r_i-r_j)^{2m_im_j}.\]
\end{defn}

We will work with the tolerant rather than the duplicant, because it is slightly more well-behaved. Of course, the tolerant and duplicant coincide for any monic polynomial, such as those needed to define the configuration space operation $\Sigma_D$.

Our main results are the following two theorems about the tolerant.

\begin{thm}[Corollary~\ref{cor:rational}]\label{thm:main rational}
    If $f(x)\in k[x]$, then $\tol(f)\in k^\times$.
\end{thm}

\begin{thm}[Theorem~\ref{thm:inversion invariant classes}]\label{thm:main inversion}
    Consider $k[x]-(x)=\{f(x)\in k[x]:f(0)\neq 0\}$. Given $f(x)\in k[x]-(x)$, denote the reciprocal polynomial by $f^*(x)$. Let
    \[T=\{f\in k[x]-(x):\tol(f)=\tol(f^*)\}.\]
    Then $T$ is a proper multiplicative subset of $k[x]-(x)$.
\end{thm}

Previously, Theorem~\ref{thm:main rational} was only known for polynomials whose roots are all elements of $k$ (which holds by definition) \cite{IMSTW24}. We prove this theorem by deriving a formula for the tolerant in terms of the irreducible factors of $f$ (Lemma~\ref{lem:disc formula general}). The primary wrinkle comes from inseparable irreducible polynomials, which we deal with in Lemma~\ref{lem:inseparable}. An amusing output of this lemma is that over any field of characteristic exponent $p$, and given any irreducible polynomial $f(x)=f_\sep(x^{p^e})$ with degree of inseparability $e$, we have
\[\tol(f)=\disc(f_\sep)^{p^e}.\]

Theorem~\ref{thm:main inversion} is interesting because the discriminant is always invariant under polynomial inversion. It is perhaps not so surprising that this property does not hold for the tolerant, but it is curious that the set of polynomials for which the tolerant is inversion invariant is multiplicative. Proposition~\ref{prop:inversion invariant} characterizes the set $T$ in terms of the roots and coefficients of its constituent elements, but it would be interesting to give a more qualitative description of $T$.

After proving these theorems, we derive a formula for $\tol(f)$ in terms of the coefficients of $f$; this is done in Section~\ref{sec:formulas}. This relies on \cite[Proposition 2.4]{DDRRS22}, where the authors give such a formula for their \emph{generalized discriminant}. In characteristic 0, \emph{loc.~cit.}~also shows that the generalized discriminant and the tolerant are equal up to a sign. We will generalize this proof to positive characteristic as well.

To conclude, we state a conjecture about a Poincar\'e--Hopf theorem for the unstable local degree at non-rational points (Conjecture~\ref{conj:poincare-hopf}). This conjecture is known to hold at rational points \cite{IMSTW24}.

\subsection*{Acknowledgments}
Both SA and BH were supported by the College of Computational, Mathematical, and Physical Sciences at BYU. We thank Rémi Prébet for bringing \cite{DDRRS22} to our attention. We thank two anonymous referees for their helpful comments.

\section{Properties of the tolerant}
Because the tolerant resembles the discriminant (and even recovers the discriminant for separable polynomials), it is natural to ask which properties of the discriminant also hold for the tolerant. We will focus on four of these properties:
\begin{enumerate}[(i)]
\item Translation invariance.
\item Homothety invariance.
\item Rationality.
\item Inversion invariance.
\end{enumerate}

\subsection{Translation and homothety invariance}
Translation and homothety invariance follow from straightforward calculations, so we will prove these properties first.

\begin{prop}[Translation invariance]
Let $f\in k[x]$. For any $\alpha\in k$, we have
\[\tol(f(x+\alpha))=\tol(f(x)).\]
\end{prop}
\begin{proof}
    Let $f=a_n\prod_{i=1}^s(x-r_i)^{m_i}$ over $\overline{k}$. Then $f(x+\alpha)=a_n\prod_{i=1}^s(x+\alpha-r_i)^{m_i}$, so
    \begin{align*}
        \tol(f(x+\alpha))&=a_n^{2n-2}\prod_{i<j}((r_i-\alpha)-(r_j-\alpha))^{2m_im_j}\\
        &=a_n^{2n-2}\prod_{i<j}(r_i-r_j)^{2m_im_j}\\
        &=\tol(f(x)).\qedhere
    \end{align*}
\end{proof}

\begin{prop}[Homothety invariance]
Let $f\in k[x]$, and let $n:=\deg(f)$. Over $\overline{k}$, write $f=a_n\prod_{i=1}^s(x-r_i)^{m_i}$. For any $\alpha\in k-\{0\}$, we have
\[\tol(f(\alpha x))=\alpha^{n^2-2n+\sum_i m_i^2}\tol(f(x)).\]
\end{prop}
\begin{proof}
    We begin by computing
    \begin{align*}
        f(\alpha x)&=a_n\prod_{i=1}^s(\alpha x-r_i)^{m_i}\\
        &=a_n\alpha^n\prod_{i=1}^s(x-\tfrac{r_i}{\alpha})^{m_i}.
    \end{align*}
    Note that $2\sum_{i<j}m_im_j=n^2-\sum_i m_i^2$. It follows that
    \begin{align*}
        \tol(f(\alpha x))&=a_n^{2n-2}\alpha^{2n^2-2n}\prod_{i<j}(\tfrac{r_i}{\alpha}-\tfrac{r_j}{\alpha})^{2m_im_j}\\
        &=a_n^{2n-2}\alpha^{n^2-2n+\sum_i m_i^2}\prod_{i<j}(r_i-r_j)^{2m_im_j}.\qedhere
    \end{align*}
\end{proof}

\begin{rem}
    If $f$ is separable (i.e.~$m_i=1$ for all $i$), then $\tol(f(\alpha x))=\alpha^{n(n-1)}\tol(f(x))$, recovering the degree of homothety of the discriminant. This is not surprising, as $\tol(f)=\disc(f)$ for separable polynomials. In general, we have
    \[n^2-2n+\sum_{i=1}^s m_i^2=n(n-1)+\sum_{i=1}^sm_i(m_i-1).\]
\end{rem}

Rationality and inversion invariance are more interesting, so we will treat each of these properties in their own subsections.

\subsection{Rationality}
If $f\in k[x]$, then $\disc(f)\in k$. This is usually proved as a consequence of the formula $\disc(f)=\res(f,f')$; this resultant is the determinant of matrix with entries determined by the coefficients of $f$ and $f'$, all of which are elements of $k$. Because we do not know any analogous resultant formulas for the tolerant, we will have to take a different approach to proving $\tol(f)\in k$. The strategy is to derive a formula for $\tol(f)$ in terms of discriminants.

Let $f(x)=a_n\prod_{i=1}^d f_i(x)^{m_i}$ be the monic irreducible factorization of $f$ over $k$. (So each $f_i(x)\in k[x]$ is monic and irreducible.) Distinct monic irreducible polynomials over a field must be coprime, so the roots of the factors $f_1,\ldots,f_d$ partition the roots of $f$. If each of the irreducible factors is separable, then we can compute $\tol(f)$ in terms of discriminants of the form $\disc(f_if_j)$ and $\disc(f_i)$.

\begin{lem}\label{lem:disc formula separable}
    Given $f(x)\in k[x]$, let $f(x)=a_n\prod_{i=1}^d f_i(x)^{m_i}$ be an irreducible factorization, so that each $f_i$ is monic and irreducible over $k$, each pair $f_i,f_j$ are coprime for $i\neq j$, and $n=\sum_{i=1}^d m_i\cdot\deg(f_i)$. Let $N=\sum_{i=1}^d m_i$. If each $f_i$ is separable, then
    \[\tol(f)=a_n^{2n-2}\prod_{i=1}^d\disc(f_i)^{m_i(2m_i-N)}\cdot\prod_{i<j}\disc(f_if_j)^{m_im_j}.\]
\end{lem}
\begin{proof}
    The proof of this lemma is given by computing the right hand side of this equation and comparing to the defining formula for $\tol(f)$. Since each irreducible factor of $f$ is separable, each root of $f$ is a root of exactly one of $f_1,\ldots,f_n$. If $r_i$ is a root of $f_i$, then $r_i$ is a root of $f$ of multiplicity $m_i$. Next, we compute
    \begin{align*}
    \disc(f_if_j)^{m_im_j}&=\disc(f_i)^{m_im_j}\cdot\disc(f_j)^{m_im_j}\cdot\prod_{\substack{f_i(r_i)=0\\f_j(r_j)=0}}(r_i-r_j)^{2m_im_j}.
    \end{align*}
    The third term is a product over all the roots of $f_i$ and $f_j$ and correctly accounts for factors of the form $(r_i-r_j)^{2m_im_j}$ in $\tol(f)$. However, two distinct roots $r_{i,1},r_{i,2}$ of $f_i$ contribute a factor of $(r_{i,1}-r_{i,2})^{2m_i^2}$ to $\tol(f)$, whereas the same pair of roots contributes a factor of $(r_{i,1}-r_{i,2})^{2m_im_j}$ to $\disc(f_i)^{m_im_j}\cdot\disc(f_j)^{m_im_j}$. We therefore need to compute the product
    \begin{align*}
        &\,\prod_{i=1}^d\disc(f_i)^{m_i(2m_i-N)}\cdot\prod_{i<j}\disc(f_i)^{m_im_j}\disc(f_j)^{m_im_j}\\
        =&\,\prod_{i=1}^d\disc(f_i)^{m_i(2m_i-N)}\cdot\prod_{i=1}^d\disc(f_i)^{m_i\sum_{j\neq i}m_j}\\
        =&\,\prod_{i=1}^d\disc(f_i)^{m_i(2m_i-N)+m_i(N-m_i)}\\
        =&\,\prod_{i=1}^d\disc(f_i)^{m_i^2}.
    \end{align*}
    The factors of $\disc(f_i)^{m_i^2}$ take the form $(r_{i,1}-r_{i,2})^{2m_i^2}$, as desired. Thus
    \begin{align*}
        &\,\prod_{i=1}^d\disc(f_i)^{m_i(2m_i-N)}\cdot\prod_{i<j}\disc(f_if_j)^{m_im_j}\\
        =&\,\prod_{i=1}^d\disc(f_i)^{m_i^2}\cdot\prod_{i<j}\prod_{\substack{f_i(r_i)=0\\ f_j(r_j)=0}}(r_i-r_j)^{2m_im_j}\\
        =&\,\prod_{\substack{f(r)=0\\f(r')=0}}(r-r')^{2mm'},
    \end{align*}
    where $m$ and $m'$ are the multiplicities of $r$ and $r'$ as roots of $f$. In particular, multiplying this product by $a^{2n-2}_n$ gives $\tol(f)$.
\end{proof}

It follows that if the irreducible factors of $f$ are separable (for example, if $k$ is perfect), then the tolerant of $f$ is rational.

\begin{cor}\label{cor:rational when separable}
    If every irreducible factor of $f$ is separable, then $\tol(f)\in k^\times$.
\end{cor}
\begin{proof}
    Each irreducible factor of $f$ is an element of $k[x]$, so $\disc(f_i),\disc(f_if_j)\in k$ by rationality of the discriminant. The claim now follows from Lemma~\ref{lem:disc formula separable}.
\end{proof}

Next, we aim to prove a similar formula for $\tol(f)$ when its irreducible factors are not all separable. Inseparable irreducible polynomials only occur over imperfect fields, which are always of positive characteristic. If $f\in k[x]$ is an inseparable and irreducible polynomial, then there exists a separable polynomial $f_\sep\in k[x]$ and an exponent $e\geq 1$ such that $f(x)=f_\sep(x^{p^e})$, where $p=\Char{k}$. If $r_1,\ldots,r_s$ are the roots of $f_\sep$ (which are necessarily distinct), then the roots of $f$ are $r_1^{p^{-e}},\ldots,r_s^{p^{-e}}$ (each with multiplicity $p^e$). This allows us to compute the tolerant of irreducible, inseparable polynomials.

\begin{lem}\label{lem:inseparable}
    Let $k$ be a field of characteristic $p$. Let $f(x)$ be a monic, irreducible, inseparable polynomial over $k$. Write $f(x)=f_\sep(x^{p^e})$, and let $r_1,\ldots,r_s$ be the roots of $f_\sep$. Then
    \begin{align*}
    \tol(f)&=\prod_{i<j}(r_i-r_j)^{2p^e}\\
    &=\disc(f_\sep)^{p^e}.
    \end{align*}
\end{lem}
\begin{proof}
    By the discussion above and the definition of the tolerant, we have
    \[\tol(f)=\prod_{i<j}(r_i^{1/p^e}-r_j^{1/p^e})^{2p^ep^e}.\]
    By the freshman's dream, we have $(r_i^{1/p^e}-r_j^{1/p^e})^{p^e}=r_i-r_j$, which yields the first equality. Finally, since $f_\sep$ is separable with roots $r_1,\ldots,r_s$, we have $\disc(f_\sep)=\prod_{i<j}(r_i-r_j)^2$.
\end{proof}

We can unify the computation of $\tol(f)$ for irreducible polynomials as follows.

\begin{defn}
    Given a field $k$, the \emph{characteristic exponent} of $k$ is $\Char{k}$ if $\Char{k}>0$, and is 1 if $\Char{k}=0$.
    
    Now let $k$ be a field of characteristic exponent $p\geq 1$. Given a polynomial $f\in k[x]$, let $e$ denote the \emph{degree of inseparability} of $f$, which is the minimal exponent $e\geq 0$ such that there exists a separable polynomial $f_\sep$ with $f(x)=f_\sep(x^{p^e})$.
\end{defn}

\begin{cor}
    If $k$ is a field of characteristic exponent $p$ and $f\in k[x]$ is a monic, irreducible polynomial, then
    \begin{equation}\label{eq:tol = disc}
    \tol(f)=\disc(f_\sep)^{p^e}.
    \end{equation}
\end{cor}
\begin{proof}
    If $\Char{k}>0$, then this is Lemma~\ref{lem:inseparable}. If $\Char{k}=0$, then all irreducible polynomials are separable, and so the $\tol(f)=\disc(f)$ in this case.
\end{proof}

Using Equation~\ref{eq:tol = disc}, we can give a formula for the tolerant of not-necessarily-separable polynomials in terms of the discriminants of their irreducible factors. We can then deduce rationality of the tolerant from this formula.

\begin{lem}\label{lem:disc formula general}
    Let $k$ be a field of characteristic exponent $p$. Given $f(x)\in k[x]$, let $f(x)=a_n\prod_{i=1}^d f_{i,\sep}(x^{p^{e_i}})^{m_i}$ be an irreducible factorization, so that each $f_i=f_{i,\sep}(x^{p^{e_i}})$ is monic and irreducible over $k$, each pair $f_i,f_j$ are coprime for $i\neq j$, and \[n=\sum_{i=1}^d m_i\cdot p^{e_i}\cdot\deg(f_{i,\sep}).\]
    Let $N=\sum_{i=1}^d m_ip^{e_i}$. Then
    \begin{equation}\label{eq:disc formula for tol}
    \tol(f)=a_n^{2n-2}\prod_{i=1}^d\disc(f_{i,\sep})^{m_ip^{e_i}(2m_ip^{e_i}-N)}\cdot\prod_{i<j}\disc(f_{i,\sep}f_{j,\sep})^{m_im_jp^{e_i+e_j}}.
    \end{equation}
\end{lem}
\begin{proof}
    The proof is analogous to that of Lemma~\ref{lem:disc formula separable}. The roots of $f_i$ take the form $r_i^{1/p^{e_i}}$, where $r_i$ is a root of $f_{i,\sep}$, each having multiplicity $m_ip^{e_i}$. The factors of the form $\disc(f_{i,\sep}f_{j,\sep})^{m_im_jp^{e_i+e_j}}$ overcount factors of the form $(r_{i,1}^{1/p^{e_i}}-r_{i,2}^{1/p^{e_i}})^{2m_i^2p^{2e_i}}$, so it remains to calculate the extent of this overcounting. The total multiplicity of such factors is
    \[2m_ip^{e_i}\sum_{j\neq i}m_jp^{e_j},\]
    so we need to correct by a multiplicity of 
    \begin{align*}
    2m_i^2p^{2e_i}-2m_ip^{e_i}\sum_{j\neq i}m_jp^{e_j}&=2m_ip^{e_i}\left(m_ip^{e_i}-\sum_{j\neq i}m_jp^{e_j}\right)\\
    &=2m_ip^{e_i}(2m_ip^{e_i}-N).\qedhere
    \end{align*}
\end{proof}

\begin{cor}\label{cor:rational}
    If $f\in k[x]$, then $\tol(f)\in k^\times$.
\end{cor}
\begin{proof}
    Each of the discriminants in Equation~\ref{eq:disc formula for tol} is an element of $k^\times$.
\end{proof}

\subsection{Inversion invariance}
We now come to the inversion of a polynomial. 

\begin{defn}
Given a polynomial $f(x)\in k[x]$ with $f(0)\neq 0$, the \emph{reciprocal} of $f$ is defined as
\[f^*(x)=x^{\deg(f)}\cdot f(\tfrac{1}{x}).\]
\end{defn}

Given a polynomial $f$ satisfying $f(0)\neq 0$, one can readily prove that $\disc(f)=\disc(f^*)$ by noting that if $r$ is a multiplicity $m$ root of $f$, then $r^{-1}$ is a multiplicity $m$ root of $f^*$. However, this same observation quickly leads to examples of polynomials whose tolerant is not inversion invariant.

\begin{ex}
    Let $f(x)=(x-2)^2(x-3)$. Then $f^*(x)=(1-2x)^2(1-3x)$, and
    \begin{align*}
        \tol(f)&=(2-3)^4=1,\\
        \tol(f^*)&=12^4(\tfrac{1}{2}-\tfrac{1}{3})^4=2^4.
    \end{align*}
\end{ex}

\begin{ex}\label{ex:not invariant}
    Let $f(x)=(x-2)^2(x+\tfrac{1}{4})$. Then $f^*(x)=(1-2x)^2(1+\tfrac{1}{4}x)$, and
    \begin{align*}
        \tol(f)&=(2+\tfrac{1}{4})^4=(\tfrac{9}{4})^4,\\
        \tol(f^*)&=(\tfrac{1}{2}+4)^4=(\tfrac{9}{2})^4.
    \end{align*}
\end{ex}

Of course, it is trivially true that $\tol(f)=\tol(f^*)$ whenever $f=f^*$; such polynomials are called \emph{palindromic}. It follows that the set of polynomials whose tolerant is inversion invariant is a non-empty proper subset $T\subset k[x]-(x)=\{f\in k[x]:f(0)\neq 0\}$. Our next goal is to characterize the set $T$. 

\begin{prop}\label{prop:inversion invariant}
    Let $f(x)=a_nx^n+\ldots+a_0$ with $a_n,a_0\neq 0$. Let $r_1,\ldots,r_s$ be the roots of $f$ with multiplicities $m_1,\ldots,m_s$. Then $\tol(f)=\tol(f^*)$ if and only if
    \begin{equation}\label{eq:inv criterion}
    \prod_{i<j}(r_ir_j)^{2m_im_j}=\left(\frac{a_0}{a_n}\right)^{2n-2}.
    \end{equation}
\end{prop}
\begin{proof}
We have $\tol(f)=\tol(f^*)$ if and only if
\begin{align*}
a_n^{2n-2}\prod_{i<j}(r_i-r_j)^{2m_im_j}&=a_0^{2n-2}\prod_{i<j}(\tfrac{1}{r_i}-\tfrac{1}{r_j})^{2m_im_j}\\
&=a_0^{2n-2}\prod_{i<j}\left(\frac{r_j-r_i}{r_ir_j}\right)^{2m_im_j}\\
&=a_0^{2n-2}\prod_{i<j}(r_ir_j)^{-2m_im_j}\cdot\prod_{i<j}(r_i-r_j)^{2m_im_j}.\qedhere
\end{align*}
\end{proof}

Here are a few classes of polynomials that satisfy Equation~\ref{eq:inv criterion}.

\begin{thm}\label{thm:inversion invariant classes}
    Let $f(x)\in k[x]$, and assume that $f(0)\neq 0$. Assume that $f$ satisfies at least one of the following criteria:
    \begin{enumerate}[(i)]
    \item $f$ is separable.
    \item $f=gh$ for some coprime $g,h\in k[x]$ satisfying Equation~\ref{eq:inv criterion}.
    \end{enumerate}
    Then $\tol(f)=\tol(f^*)$.
\end{thm}
\begin{proof}
    We will treat each case in turn.
    \begin{enumerate}[(i)]
    \item If $f$ is separable, then $f^*$ is also separable. It then follows that $\tol(f)=\disc(f)=\disc(f^*)=\tol(f^*)$.
    \item Let $g(x)=g_ax^a+\ldots+g_0$ with roots $r_1,\ldots,r_c$ of multiplicities $m_1,\ldots,m_c$ and $h(x)=h_bx^b+\ldots+h_0$ with roots $s_1,\ldots,s_d$ of multiplicities $n_1,\ldots,n_d$. Then $gh(x)=g_ah_bx^{a+b}+\ldots+g_0h_0$, and the set of roots (with their multiplicities) of $gh$ is the union of the sets of roots of $g$ and $h$ (with their multiplicities). Denote the roots of $f=gh$ by $z_i$, with multiplicity $\ell_i$. Since $g$ and $h$ both satisfy Equation~\ref{eq:inv criterion}, it follows that
    \begin{align*}
        \prod_{i<j}(z_iz_j)^{2\ell_i\ell_j}&=\prod_{1\leq i<j\leq c}(r_ir_j)^{2m_im_j}\cdot\prod_{1\leq i<j\leq d}(s_is_j)^{2n_in_j}\cdot\prod_{\substack{1\leq i\leq c\\ 1\leq j\leq d}}(r_is_j)^{2m_in_j}\\
        &=\left(\frac{g_0}{g_a}\right)^{2a-2}\cdot\left(\frac{h_0}{h_b}\right)^{2b-2}\cdot\prod_{i=1}^c r_i^{2m_ib}\cdot\prod_{j=1}^d s_j^{2n_ja}.
    \end{align*}
    By Vieta's formulas, we have $\prod_i r_i^{m_i}=\pm\frac{g_0}{g_a}$ and $\prod_j s_j^{n_j}=\pm\frac{h_0}{h_b}$, so $\prod_i r_i^{2m_ib}=(\frac{g_0}{g_a})^{2b}$ and $\prod_j s_j^{2n_ja}=(\frac{h_0}{h_b})^{2a}$. Thus
    \begin{align*}
        \prod_{i<j}(z_iz_j)^{2\ell_i\ell_j}&=\left(\frac{g_0}{g_a}\right)^{2a+2b-2}\cdot\left(\frac{h_0}{h_b}\right)^{2a+2b-2}\\
        &=\left(\frac{g_0h_0}{g_ah_b}\right)^{2(a+b)-2},
    \end{align*}
    so $f$ satisfies Equation~\ref{eq:inv criterion}.\qedhere
    \end{enumerate}
\end{proof}

Theorem~\ref{thm:inversion invariant classes} (ii) states that the set $T$ of polynomials for which the tolerant is inversion invariant is a multiplicative set. We note that not all elements of $T$ are characterized by Theorem~\ref{thm:inversion invariant classes}. For a simple example, $f(x)=(x-1)^n$ and its reciprocal both have tolerant 1, as their leading coefficients both square to 1 and the product of pairs of roots is an empty product. 

The previous example might lead one to guess that any power of a separable polynomial belongs to $T$, but this guess is incorrect. Indeed, together with Theorem~\ref{thm:inversion invariant classes} and the fundamental theorem of algebra, such a statement would imply that $T=k[x]$ whenever $k$ is a perfect field. Such a conclusion would contradict Example~\ref{ex:not invariant}. We can also directly construct an example of a polynomial that is a power of a separable polynomial and does not belong to $T$.

\begin{ex}
    Let $f(x)=(x-1)^2(x-2)^2$. Then $f^*(x)=(1-x)^2(1-2x)^2$, and
    \begin{align*}
        \tol(f)&=(1-2)^8=1,\\
        \tol(f^*)&=4^6(1-\tfrac{1}{2})^8=2^4.
    \end{align*}
\end{ex}

\section{Intrinsic formula for the tolerant}\label{sec:formulas}
So far, all known formulas for the tolerant require some information about its irreducible factorization. One of the key advantages of the discriminant is that it provides information about the irreducible factors (namely, whether there is a multiple root) strictly in terms of the coefficients of $f(x)=a_nx^n+\ldots+a_0$. This can be seen by the resultant formula
\[\disc(f)=\frac{(-1)^{\binom{n}{2}}}{a_n}\res(f,f').\]
We are interested in deriving an analogous formula for the tolerant. We had left such a formula as an open question in our first preprint of this article, but Rémi Prébet directed our attention to \cite{DDRRS22}. Using the ideas in this article, we can give the desired resultant formula for $\tol(f)$.

\begin{defn}
    For each $r\geq 0$, let \[D^r:k[x]\to k[x]\] be the $k$-linear map determined by
    \[D^rx^n=\begin{cases}\binom{n}{r}x^{n-r} & n\geq r,\\
    0 & n<r.
    \end{cases}\]
    The map $D^r$ is called the \emph{$r\textsuperscript{th}$ Hasse derivative}. Given $f\in k[x]$, we will write
    \[f^{[r]}:=D^rf.\]
\end{defn}

Note that if $\Char{k}=0$, then $D^r=\frac{1}{r!}\frac{\d^r}{\d{x}^r}$. Hasse derivatives are defined in a manner that provides an analog of Taylor's theorem over fields of arbitrary characteristic \cite{Gol03}. In particular, for any $f\in k[x]$ and $\alpha\in\overline{k}$, one has
\begin{equation}\label{eq:taylor}
f(x+\alpha)=\sum_{i=0}^{\deg(f)}f^{[i]}(\alpha)x^i.
\end{equation}

\begin{notn}
    Given a ring $R$ and a polynomial $f\in R[z]$, let $\lc_z(f)$ and $\tc_z(f)$ denote the coefficients of the non-zero terms in $f$ of largest and smallest degree, respectively. We will refer to $\lc_z(f)$ and $\tc_z(f)$ as the \emph{leading} and \emph{trailing coefficients} of $f$.
\end{notn}

\begin{defn}\cite[Definition 2.3]{DDRRS22}\label{def:gen disc}
    Given $f\in k[x]$, the \emph{generalized discriminant} of $f$ is defined as
    \[\gdisc(f):=\lc_x(f)^{-1}\cdot\tc_u\left(\res_x\left(f,\sum_{i=1}^{\deg(f)}u^{i-1}\cdot f^{[i]}\right)\right).\]
\end{defn}

We now prove that $\gdisc(f)=\pm\tol(f)$ over any field, generalizing \cite[Proposition 2.4]{DDRRS22}.

\begin{thm}\label{thm:resultant}
   We have $\gdisc(f)=(-1)^{\binom{\deg(f)}{2}}\tol(f)$. 
\end{thm}
\begin{proof}
    Let $g,h\in k[x]$ of degrees $m$ and $n$. Let $r_1,\ldots,r_s\in\overline{k}$ denote the roots of $g$, and let $m_i$ denote the multiplicity of $r_i$ as a root of $g$. It is a standard fact about resultants that
    \[\res_x(g,h)=\lc_x(g)^n\cdot\prod_{i=1}^s h(r_i)^{m_i}.\]
    Applying this fact to $g=f$ and $h=\sum_{i=1}^{\deg(f)}u^{i-1}\cdot f^{[i]}$, we have
    \[\res_x\left(f,\sum_{i=1}^{\deg(f)}u^{i-1}\cdot f^{[i]}\right)=\lc_x(f)^{\deg(f)-1}\cdot\prod_{j=1}^s\sum_{i=1}^{\deg(f)} u^{i-1}\cdot f^{[i]}(r_j)^{m_i}.\]
    Note that $f^{[i]}(r_j)=0$ for $i<m_j$ and
    \begin{equation}\label{eq:sub}
    f^{[m_j]}(r_j)=\lc_x(f)\cdot\prod_{i\neq j}(r_j-r_i)^{m_i}\neq 0
    \end{equation}
    by Equation~\ref{eq:taylor}. This allows us to compute the trailing coefficient
    \[\tc_u\left(\lc_x(f)^{\deg(f)-1)}\cdot\prod_{j=1}^s\sum_{i=1}^{\deg(f)}u^{i-1}\cdot f^{[i]}(r_j)^{m_i}\right)=\lc_x(f)^{\deg(f)-1}\cdot\prod_{j=1}^s f^{[m_j]}(r_j)^{m_j}.\]
    Combined with Equation~\ref{eq:sub}, we find that
    \begin{align*}
        \gdisc(f)&=\lc_x(f)^{2\deg(f)-2}\cdot\prod_{i=1}^s\prod_{j\neq i}(r_j-r_i)^{m_im_j}\\
        &=(-1)^{\binom{\deg(f)}{2}}\lc_x(f)^{2\deg(f)-2}\cdot\prod_{i<j}(r_i-r_j)^{2m_im_j}\\
        &=(-1)^{\binom{\deg(f)}{2}}\tol(f).\qedhere
    \end{align*}
\end{proof}

The upshot of Theorem~\ref{thm:resultant} is that we can compute the tolerant of a polynomial purely in terms of its coefficients via the resultant formula given in Definition~\ref{def:gen disc}.

\section{Unstable Poincar\'e--Hopf at non-rational points}
To conclude this article, we will apply Corollary~\ref{cor:rational} to state a conjecture on removing the rationality assumption from \cite[Proposition 6.5]{IMSTW24}.

\begin{conj}\label{conj:poincare-hopf}
    Let $f/g:\mb{P}^1\to\mb{P}^1$ be a pointed rational function with vanishing locus $D=\{r_1,\ldots,r_n\}\subset\mb{A}^1_k$. Let $m_i(x)\in k[x]$ denote the (monic) minimal polynomial of $r_i$ for each $i$. At each zero, let the unstable local degree be given by $\deg_{r_i}(f/g)=(\beta_i,d_i)\in\GW(k)\times_{k^\times/(k^\times)^2}k^\times$, where $\GW(k)$ denotes the Grothendieck--Witt group over $k$. Then we have the following decomposition of the unstable degree of $f/g$:
    \[\deg(f/g)=\left(\bigoplus_{i=1}^n\beta_i,\prod_{i=1}^n d_i\cdot\tol\left(\prod_{i=1}^n m_i(x)\right)\right).\]
\end{conj}

More generally, we expect that
\[\bigoplus_D(\beta_i,d_i)=\left(\bigoplus_{i=1}^n\beta_i,\prod_{i=1}^n d_i\cdot\tol\left(\prod_{i=1}^n m_i(x)\right)\right)\]
for any $(\beta_1,d_1),\ldots,(\beta_n,d_n)\in\GW(k)\times_{k^\times/(k^\times)^2}k^\times$. This more general statement should follow from Conjecture~\ref{conj:poincare-hopf} by applying the same inductive method used in \cite{IMSTW24}. In \emph{op.~cit.}, the known case of Conjecture~\ref{conj:poincare-hopf} forms the base case of the induction.

In order to prove Conjecture~\ref{conj:poincare-hopf}, one will probably first need to derive a formula for the unstable local degree at non-rational points. At rational points, the unstable local degree is given by something called the \emph{local Newton matrix}. The local Newton matrix has only been defined at rational points. To extend the definition to a non-rational point $r$, and to employ the change-of-basis techniques used to prove \cite[Proposition 6.5]{IMSTW24}, one needs a suitable choice of $k$-basis for the residue field $k(r)$. A promising candidate for such a basis is the Horner basis \cite{BM23}.

\section{Conclusion}
The discriminant is a classical invariant of polynomials. The utility of the discriminant is that it vanishes on polynomials with repeated roots but can be computed without knowing the roots of the given polynomial. In this article, we investigated the \emph{tolerant}, which coincides with the discriminant for polynomials with no repeated roots. In contrast to the discriminant, the tolerant never vanishes. 

Like the discriminant, the tolerant is invariant under translations and homothety, and it is always valued in the field of definition of the given polynomial. Unlike the discriminant, the tolerant is not always inversion invariant. Interestingly, the set of polynomials for which the tolerant \emph{is} inversion invariant is multiplicative.

The tolerant is almost identical to two other notions in the literature. The first, which was the inspiration for this article, is the \emph{duplicant}, which naturally arises in unstable motivic homotopy theory. The second is the \emph{generalized discriminant}, which was discovered in \cite{DDRRS22}. The tolerant, duplicant, and generalized discriminant all differ from each other by multiplication by a scalar.

\bibliography{tolerants}{}
\bibliographystyle{alpha}
\end{document}